\documentclass[a4paper,11pt]{article} 
\usepackage[latin1]{inputenc}
\usepackage{amsfonts}
\usepackage{amsmath}
\usepackage{mathrsfs}
\usepackage{graphicx}
\usepackage{latexsym}
\usepackage[all]{xy}
\usepackage{amsthm}
\usepackage{amssymb}

\title{Chordality properties and hyperbolicity on graphs}
\author{A. Martínez-Pérez\footnote{{The author was partially supported by MTM-2012-30719.}}}

\newtheorem{definicion}{Definition}[section]

\newtheorem{prop}[definicion]{Proposition}
\newtheorem{lema}[definicion]{Lemma}
\newtheorem{obs}[definicion]{Remark}
\newtheorem{teorema}[definicion]{Theorem}

\begin{document}

\maketitle

\begin{abstract} Let $G$ be a graph with the usual shortest-path metric. A graph is $\delta$-hyperbolic if for every geodesic triangle $T$, any side of $T$ is contained in a $\delta$-neighborhood of the union of the other two sides. A graph is chordal if every induced cycle has at most three edges. In this paper we study the relation between the hyperbolicity of the graph and some chordality properties which are natural generalizations of being chordal. We find chordality properties that are weaker and stronger than being $\delta$-hyperbolic. Moreover, we obtain a characterization of being hyperbolic on terms of a chordality property on the triangles.
\end{abstract}

\begin{footnotesize}
\end{footnotesize}

\begin{footnotesize}
\end{footnotesize}

\section{Introduction}

The theory of Gromov hyperbolic spaces was introduced by M. Gromov for the study of finitely generated groups (see \cite{Gr}). Since then, this theory has been developed from a geometric point of view to the extent of making hyperbolic spaces an important class of metric spaces to be studied on their own (see, for example, \cite{B-H,Bu-Bu,BS,G-H,Vai}). 
In the last years, Gromov hyperbolicity has been intensely studied in graphs (see \cite{BRS2,BRS,BRSV,CPRS,CRSV,CDEHV,K50,Ha,MRSV,PeRSV,PRSV,PRT1,PRT2,PT,RS,RSTY,RSVV,Si,T}). Gromov hyperbolicity, specially in graphs, has found applications in different areas such as phylogenetics (see \cite{DHH,DMT}), complex networks (see \cite{CMN,KP,Sha1,Sha2}) or the secure transmission of information and virus propagation on networks (see \cite{K21,K22}).

Given a metric space $(X,d)$,  a \emph{geodesic} from $x\in X$ to $y\in X$ is an isometry, $\gamma$, from a closed interval $[0,l]\subset \mathbb{R}$ to $X$ such that $\gamma(0)=x$, $\gamma(l)=y$. We will also call geodesic to the image of $\gamma$. $X$ is a geodesic metric space if for every $x,y\in X$ there exists a geodesic joining $x$ and $y$; any of these geodesics will be denoted as $[xy]$ although this notation is ambiguous since geodesics need not be unique.

By a \emph{metric graph} we mean a graph $G$ equipped with a length metric by considering every edge $e\in E(G)$ as isometric to an interval $[0,l_e]$. Thus, the interior points of the edges are also considered points in $G$. Then, for any pair of points in $x,y \in G$,  the distance $d(x,y)$ will be the length of the shortest path in $G$ joining $x$ and $y$. In this paper we will assume that the graphs are connected and locally finite (i.e. each ball intersects a finite number of edges) and there is no restriction on the length of the edges. In this case, the metric graph is a geodesic metric space.

There are several definitions of Gromov $\delta$-hyperbolic space  which are equivalent although the 
constant $\delta$ may appear multiplied by some constant (see \cite{BS}). We are going to use the characterization of 
Gromov hyperbolicity for geodesic metric spaces given by the Rips condition on the geodesic triangles. 
If $X$ is a geodesic metric space and $x_1,x_2,x_3\in X$, the union
of three geodesics $[x_1 x_2]$, $[x_2 x_3]$ and $[x_3 x_1]$ is called a
\emph{geodesic triangle} and will be denoted by $T=\{x_1,x_2,x_3\}$. $T$ is $\delta$-{\it thin} 
if any side of $T$ is contained in the
$\delta$-neighborhood of the union of the two other sides.  
The space $X$ is $\delta$-\emph{hyperbolic} if every geodesic triangle in $X$ is $\delta$-thin. We
denote by $\delta(X)$ the sharp hyperbolicity constant of $X$, i.e.
$\delta(X):=\inf\{\delta \, | \, \text{every triangle in $X$ is $\delta$-thin}\}.$ We say that $X$ is \emph{hyperbolic} if $X$ is
$\delta$-hyperbolic for some $\delta \geq 0$. A triangle with two
identical vertices is called a ``bigon''.

A graph $G$ is \emph{chordal} if every induced cycle has at most three edges. In \cite{BKM}, the authors prove that chordal graphs are hyperbolic  giving an upper bound for the hyperbolicity constant. In \cite{WZ}, Wu and Zhang extend this result for a generalized version of chordality. They prove that $k$-chordal graphs are hyperbolic where a graph is $k$-chordal if every induced cycle has 
at most $k$ edges. In \cite{B}, the authors define the more general properties of being $(k,m)$-edge-chordal and $(k,\frac{k}{2})$-path-chordal and prove that every $(k,m)$-edge-chordal graph is hyperbolic and that every hyperbolic graph is $(k,\frac{k}{2})$-path-chordal. Herein, we continue this work and define being $\varepsilon$-densely $(k,m)$-path-chordal and $\varepsilon$-densely $(k,m)$-path-chordal obtaining that (see Theorems \ref{Th: edge-path}, \ref{Th: path-hyp} and \ref{Th: hyp-path})

\[(k,m)\mbox{-edge-chordal} \Rightarrow  \varepsilon\mbox{-densely } (k,m)\mbox{-path-chordal} \Rightarrow \delta\mbox{-hyperbolic} \]
and
\[\delta\mbox{-hyperbolic} \Rightarrow  \varepsilon\mbox{-densely } k\mbox{-path-chordal} \Rightarrow k\mbox{-path-chordal}\]

We also provide examples showing that all these implications are strict, this is, the converse is not true. 

Moreover, we give a characterization of hyperbolicity in terms of a chordality property on the triangles: Theorem \ref{Th: carac} states that
a metric graph $G$ is $\delta$-hyperbolic if and only if $G$ is $\varepsilon$-densely $(k,m)$-path-chordal on the triangles.

The properties and implications studied in this paper are summarized in Figure \ref{Figure: Chordal}.

\begin{figure}[ht]
\centering
\includegraphics[scale=0.4]{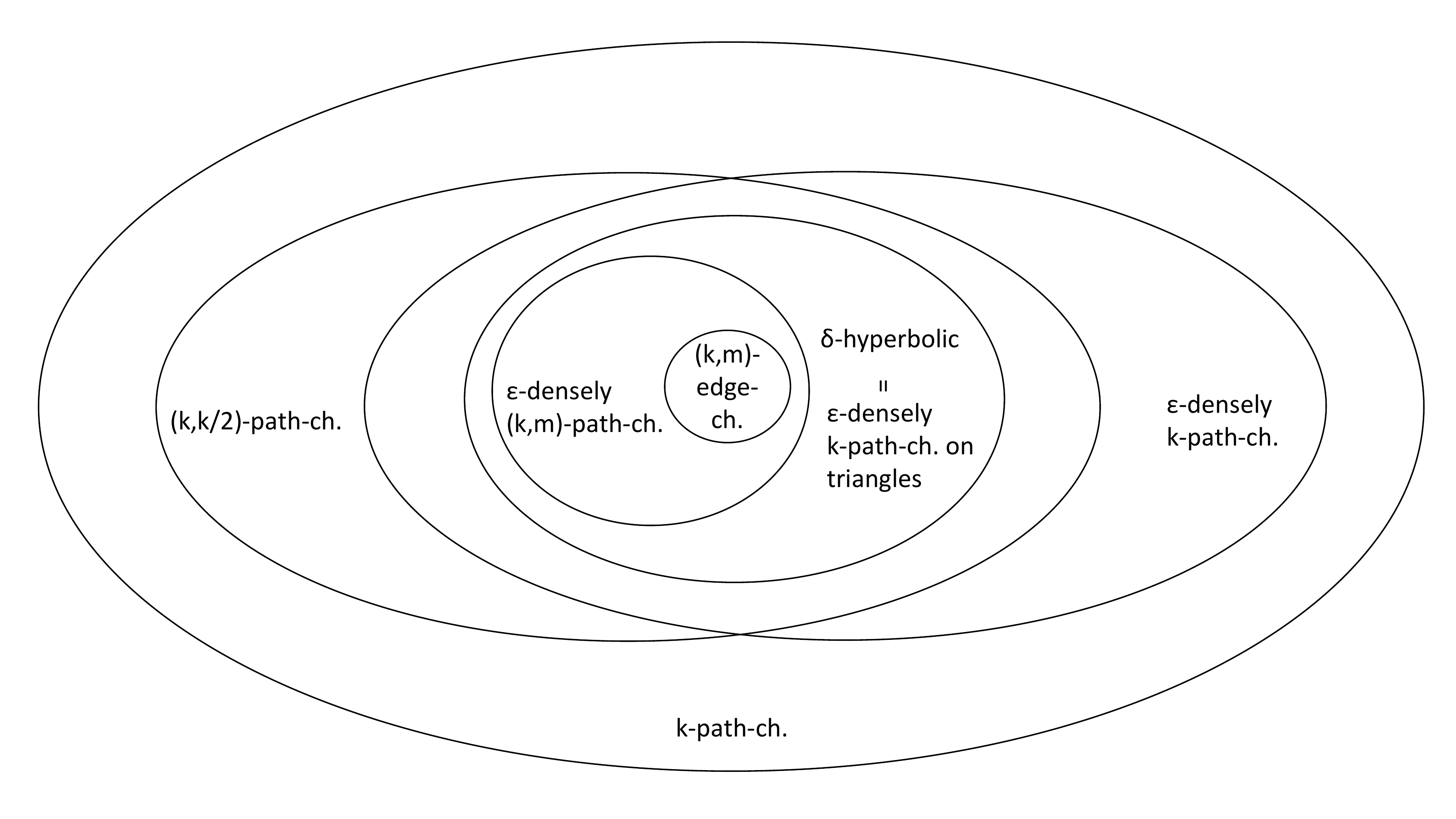}
\caption{Chordality properties and $\delta$-hyperbolicity on graphs.}
\label{Figure: Chordal}
\end{figure}

\section{$\varepsilon$-densely path-chordal graphs}

Consider a metric graph $(G,d)$. By a cycle in a graph we mean a simple closed curve, this is, a path defined by a sequence of vertices which are all different except for the first one and the last one which are the same. A shortcut in a cycle $C$ is a path $\sigma$ joining two vertices $p,q$ in $C$ such that $L(\sigma)<d_C(p,q)$ where $d_C$ denotes the length metric on $C$.  

\begin{definicion} A shortcut $\sigma$ in a cycle $C$ is \emph{strict} if $\sigma \cap C=\{p,q\}$. In this case, we say that $p$, $q$ are 
\emph{shortcut vertices} in $C$ associated to $\sigma$. 
\end{definicion}

Let us recall two definitions from \cite{B}:

Given two constants $k,m \geq 0$, a metric graph $G$ is $(k,m)$-\emph{edge-chordal} if for any cycle $C$ in $G$ with length $L(C)\geq k$ there exists an edge-shortcut $e$ with length $L(e)\leq m$. The graph $G$ is \emph{edge-chordal} if there
exist constants $k,m\geq 0$ such that $G$ is $(k,m)$-edge-chordal. Notice that for a graph with edges of length 1, being chordal is equivalent to being $(4,1)$-edge-chordal and being $k$-chordal in the sense of \cite{WZ} is equivalent to being $(k+1,1)$-edge-chordal.

A metric graph $G$ is $(k,\frac{k}{2})$-\emph{path-chordal} if for any cycle $C$ in $G$ with $L(C)\geq k$ there exists a
shortcut $\sigma$ of $C$ such that $L(\sigma)\leq \frac{k}{2}$. Notice that in \cite{B} this is called simply $k$-path-chordal. However, new definitions are introduced herein and this property has been renamed. 

Given a metric space $(X,d)$ and any $\varepsilon>0$, a subset $A\subset X$ is $\varepsilon$-\emph{dense} if for every $x\in X$ there exist some $a\in A$ such that $d(a,x)<\varepsilon$.

\begin{definicion} A metric graph $(G,d)$ is \emph{$\varepsilon$-densely $(k,m)$-path-chordal} if for every cycle $C$ with length $L(C)\geq k$, there exist strict shortcuts $\sigma_1,...,\sigma_r$ with $L(\sigma_i)\leq m$ and such that their associated shortcut vertices define an $\varepsilon$-dense subset in $(C,d_C)$.
\end{definicion}

\begin{teorema}\label{Th: edge-path} If a metric graph $G$ is $(k,m)$-edge-chordal, then it is $\varepsilon$-densely $(k,m)$-path-chordal. 
\end{teorema}

\begin{proof} Clearly, being $(k,m)$-edge-chordal implies being $(k,m)$-path-chordal. Thus, it suffices to check that there is some 
$\varepsilon>0$ such that,  on every cycle $C$ with length at least $k$, the shortcut vertices associated to the edge-shortcuts with length at most $m$ are $\varepsilon$-dense. 

Let $\varepsilon:=\frac{k}{2}$ and let $p$ be any point in $C$. Since $G$ is $(k,m)$-edge-chordal, there is an edge $xy$ with $L(xy)\leq m$ which is a shortcut in $C$. If either $x$ or $y$ is in $B_C(p,\frac{k}{2})$, we are done. Otherwise, let us define a cycle $C_1$ by joining the path in $C$ from $x$ to $y$ containing $p$ and the edge $xy$. Since $x,y\notin B_C(p,\frac{k}{2})$, $C_1$ has length at least $k$ and there is another edge, $x'y'$, with $L(x'y')\leq m$ which defines a shortcut in $C_1$. Since $xy$ is an edge, $x',y'\in C$ and $x'y'$ is also an edge-shortcut in $C$. If either $x'$ or $y'$ is in 
$B_C(p,\frac{k}{2})$, we are done. This process can be repeated and, since the graph is finite, we will finally obtain an edge-shortcut in $C$  such that one of its vertices is contained in $B_C(p,\frac{k}{2})$ finishing the proof.
\end{proof}

The converse is not true. Let $P_3$ be the path graph with (adjacent) vertices $v_1, v_2, v_3$, and $G$ the Cartesian product graph
$G=\mathbb{Z}\Box P_3$ with $L(e)=1$ for every $e\in E(G)$. As it was shown in \cite{B}, $G$ is not edge-chordal although it is $\frac{5}{2}$-hyperbolic and $(5,\frac{5}{2})$-path chordal.

Let us see that $G$ is 3-densely $(5,2)$-path-chordal. Let $C$ be any cycle in $G$ with $L(C)\geq 5$ and let $(z_i,v_j)$ be any vertex in $C$. It suffices to check that there is a shortcut vertex in $B_C\big((z_i,v_j),\frac{5}{2}\big)$ associated to a shortcut with length at most 2.

Case 1. If $(z_{i-1},v_j)(z_i,v_j),(z_{i},v_j)(z_{i+1},v_j)\in C$, then it is immediate to see that there exist some vertex $(z_i,v_k)\in C$ with $k\neq j$. Then, the geodesic $[(z_i,v_j)(z_i,v_k)]$ contains a strict shortcut $\sigma$ with $L(\sigma)\leq 2$ such that $(z_i,v_j)$ is a shortcut vertex associated to $\sigma$.

Case 2. If $(z_{i},v_{j-1})(z_i,v_j),(z_{i},v_{j})(z_{i},v_{j+1})\in C$, then $j=2$ and either $(z_{i-1},v_{3})(z_i,v_3),(z_{i-1},v_{1})(z_i,v_1)\in C$ or $(z_{i},v_{3})(z_{i+1},v_3),(z_{i},v_{1})(z_{i+3},v_1)\in C$. Suppose that $(z_{i-1},v_{3})(z_i,v_3),(z_{i-1},v_{1})(z_i,v_1)\in C$ (the other case is equivalent by symmetry). Therefore, if the geodesic $\gamma=[(z_{i-1},v_{3})(z_{i-1},v_{1})]$ is contained in $C$, then 
$\sigma=(z_{i-1},v_{2})(z_i,v_2)$ is a shortcut with $L(\sigma)=1$ and $(z_i,v_2)$ is a shortcut vertex associated to $\sigma$. If $\gamma \not \subset C$, then it contains a strict shortcut $\sigma$ with $L(\sigma)\leq 2$ and either $(z_{i-1},v_{3})$ or $(z_{i-1},v_{1})$ is a shortcut vertex associated to $\sigma$. Thus, there is a shortcut vertex associated to a shortcut $\sigma$ with $L(\sigma)\leq 2$ in 
$B_C\big((z_i,v_j),\frac{5}{2}\big)$.

Case 3. $(z_i,v_j)$ is a corner vertex in $C$. Let us suppose that it is an upper-left corner, this is 
$(z_{i},v_{j-1})(z_i,v_j),(z_{i},v_j)(z_{i+1},v_j)\in C$ (the other cases are equivalent by symmetry).

Case 3.1.  If $(z_{i+1},v_j)(z_{i+1},v_{j-1})\in C$, then, since $L(C)\geq 5$, $\sigma=(z_{i},v_{j-1})(z_{i+1},v_{j-1})$ is a shortcut and 
$(z_{i},v_{j-1})$ is a shortcut vertex associated to $\sigma$. If $(z_{i},v_{j-1})(z_{i+1},v_{j-1})\in C$, then,  since $L(C)\geq 5$, 
$\sigma=(z_{i+1},v_{j})(z_{i+1},v_{j-1})$ is a shortcut and $(z_{i+1},v_{j})$ is a shortcut vertex associated to $\sigma$. Therefore, there is a shortcut vertex associated to a shortcut $\sigma$ with $L(\sigma)=1$ in $B_C\big((z_i,v_j),\frac{5}{2} \big)$

Case 3.2. Suppose $(z_{i+1},v_j)(z_{i+1},v_{j-1}), (z_{i},v_{j-1})(z_{i+1},v_{j-1})\notin C$. 

Case 3.2.1. If $(z_{i},v_{j-1})(z_{i},v_{j-2})\in C$, it is immediate to check that $(z_{i},v_{j-2})(z_{i+1},v_{j-2})\in C$. Now, if 
$(z_{i+1},v_{j-2})(z_{i+1},v_{j-1})\in C$, then $\sigma=(z_i,v_{j-1})(z_{i+1},v_{j-1})$ is a shortcut with $L(\sigma)=1$ and  $(z_i,v_{j-1})$ is a shortcut vertex associated to $\sigma$ and contained in $B_C\big((z_i,v_j),\frac{5}{2} \big)$. Otherwise, if $(z_{i+1},v_{j-2})(z_{i+1},v_{j-1})\notin C$, then $\sigma=[(z_{i+1},v_{j-2})(z_{i+1},v_{j})]$ is a strict shortcut with $L(\sigma)=2$ and $(z_{i+1},v_{j})$ is a shortcut vertex associated to $\sigma$ and contained in $B_C\big((z_i,v_j),\frac{5}{2} \big)$.

Case 3.2.2. If $(z_{i},v_{j-1})(z_{i},v_{j-2})\notin C$, then the only case left is that $(z_{i-1},v_{j-1})(z_{i},v_{j-1})\in C$. 
Then, either $(z_i,v_{j+1})\in C$ or $(z_i,v_{j-2})\in C$.

If $(z_i,v_{j+1})\in C$, then $\sigma=(z_{i},v_{j+1})(z_{i},v_{j})$  is a shortcut $\sigma$ with $L(\sigma)=1$ and $(z_i,v_j)$ is a shortcut vertex associated to $\sigma$. 

If $(z_i,v_{j-2})\in C$, then $\sigma=(z_i,v_{j-1})(z_i,z_{j-2})$ is a shortcut with $L(\sigma)=1$ and $(z_i,z_{j-1})$ is a shortcut vertex associated to $\sigma$ and contained in $B_C\big((z_i,v_j),\frac{5}{2} \big)$.

\begin{teorema}\label{Th: path-hyp} If $G$ is $\varepsilon$-densely $(k,m)$-path-chordal, then $G$ is $\delta$-hyperbolic. Moreover,  $\delta(G)\leq max\{\frac{k}{4}, \varepsilon+m\}$.
\end{teorema}

\begin{proof} Suppose that $G$ is $\varepsilon$-densely $(k,m)$-path-chordal. Consider any geodesic triangle $T=\{x,y,z\}$ in $G$. If $L(T)<k$, it follows that every side of the triangle has length at most $\frac{k}{2}$. Therefore, the hyperbolic constant is at most $\frac{k}{4}$. If $L(T)\geq k$, let us prove that $T$ is $(\varepsilon+m)$-hyperbolic. Consider any point $p\in T$ and let us assume that $p\in [xy]$. If $d(p,x)< \varepsilon+m$ or $d(p,y)< \varepsilon+m$, we are done. Otherwise, there is a shortcut vertex $x_i$ such that $d(x_i,p)<\varepsilon$ and a shortcut $\sigma_i$, with $x_i\in \sigma_i$ and $L(\sigma_i)\leq m$. Since $[xy]$ is a geodesic, $\sigma_i$ is not contained in $[xy]$ and $d(p, [xz]\cup [yz])<\varepsilon+m$.
\end{proof}

The converse is not true. To build an example of a $\delta$-hyperbolic graph which is not $\varepsilon$-densely $(k,m)$-path-chordal let us recall here the construction of the \emph{hyperbolic approximation} of a metric space introduced by S. Buyalo and V. Schroeder in \cite{BS}. The hyperbolic approximation of a metric space is a special kind
of hyperbolic cone, see \cite{BoS}, which is defined in general for non-necessarily bounded metric spaces. 

A subset $A$ in a metric space $Z$ is called \emph{r-separated},
$r>0$, if $d(a,a')\geq r$ for any distinct $a,a'\in A$. Note that
if $A$ is maximal with this property, then the union $\cup_{a\in
A} B_r(a)$ covers $Z$.

A \emph{hyperbolic approximation} of a metric space $Z$ is a graph
$X$ which is defined as follows. Fix a positive $r\leq
\frac{1}{6}$ which is called the \emph{parameter} of $X$. For
every $k\in \mathbb{Z}$, let $A_k\in Z$ be a maximal
$r^k$-separated set. For every $a\in A_k$, consider the ball
$B(a,2r^k)\subset Z$. 
Let us fix  the set $V$ as the union, for $k\in
\mathbb{Z}$, of the set of balls $B(a,2r^k)$, $a\in A_k$. 
 Let us denote the corresponding ball simply by $B(v)$. 

Let $V$ be the vertex set of the graph $X$.
Vertices $v,v'$ are connected by an edge if and only if they
either belong to the same level, $V_k$, and the closed balls
$\bar{B}(v),\bar{B}(v')$ intersect, $\bar{B}(v)\cap
\bar{B}(v')\neq \emptyset$, or they lie on neighboring levels
$V_k,V_{k+1}$ and the ball of the upper level, $V_{k+1}$, is
contained in the ball of the lower level, $V_k$.

An edge $vv'\subset X$ is called \emph{horizontal} if its
vertices belong to the same level, $v,v'\in V_k$ for some $k\in
\mathbb{Z}$. Other edges are called \emph{radial}. Consider the
path metric on $X$ for which every edge has length 1. 
There is a natural level function $l:V \to \mathbb{Z}$ defined by $l(v)=k$
for $v\in V_k$. 

Note that any (finite or infinite) sequence $\{v_k\}\in V$ such
that $v_kv_{k+1}$ is a radial edge for every $k$ and such that the level
function $l$ is monotone along $\{v_k\}$, is the vertex sequence
of a geodesic in $X$. Such a geodesic is called \emph{radial}.

\begin{prop}\cite[Proposition 6.2.10]{BS}\label{Prop: hyp} A hyperbolic approximation of any metric space is a
geodesic $3$-hyperbolic space.
\end{prop}

The following technical lemma is very useful to understand the geodesics in $X$.

\begin{lema}\cite[Lemma 6.2.6]{BS}\label{Lemma: Geod} Any vertices $v,v'\in V$ can be connected in $X$ by a
geodesic which contains at most one horizontal edge. If there is
such an edge, then it lies on the lowest level of the geodesic.
\end{lema}

Now, let us consider the following hyperbolic approximation of the Euclidean real line.

Let $r=\frac{1}{6}$ and $A_k:=\{mr^k \, : \, m\in \mathbb{Z}\}$. By Proposition \ref{Prop: hyp}, the resultant hyperbolic approximation $G$ is $\delta$-hyperbolic. Let us see that $G$ not $\varepsilon$-densely $(k,m)$-path-chordal. 

Consider the cycle $C_n\in G$ defined as follows. First, consider the vertices $v_0,v_1,...,v_{6^n}$ in $V_0$ such that $B(v_i)=B(i,2)$ for every $0\leq i\leq 6^n$ and the horizontal edges $v_{i-1}v_i$ for every $1\leq i\leq 6^n$.
Also, for every $0 < k < n$ consider the vertices $v'_{k},v''_k$ in $V_{-k}$ such that $B(v'_{k})=B(6^k, 2\cdot 6^k)$ and $B(v''_k)=B(6^n-6^k, 2\cdot 6^k)$. Then, consider the radial edges $v_0v'_1$, $v_{6^n}v''_1$ and $v'_{k-1}v'_k$, $v''_{k-1}v''_k$ for every $1<k<n$.
Finally, to complete the cycle, consider the horizontal edge $v'_{n-1}v''_{n-1}$.

Let $\gamma$ be the path in $C_n$ from $v_0$ to $v_{6^n}$ given by the radial geodesic from $v_0$ to $v'_{n-1}$, the horizontal edge $v'_{n-1}v''_{n-1}$ and the radial geodesic from $v''_{n-1}$ to $v_{6^n}$. Let us prove that $\gamma$ is a geodesic. 

Notice that $L(\gamma)=2n-1$. Suppose that $\gamma'$ is a geodesic in $X$ from $v_0$ to $v_{6^n}$ with $L(\gamma')<2n-1$. By Lemma \ref{Lemma: Geod} we may assume that $\gamma'$ contains at most one horizontal edge  which lies in the lowest level, $-m$. Moreover, we may assume that $\gamma'$ begins with $m$ radial edges from level 0 to level $-m$, then it may have one horizontal edge or not, and then it has $m$ radial edges from level $-m$ to level 0.

Since  $L(\gamma')<2n-1$, then either there is no horizontal edge or $m<n-1$. 

Now, notice that for any horizontal edge $vv'$ with $B(v)=B(i\cdot 6^m, 2\cdot 6^m)$ and $B(v')=B(j\cdot 6^m, 2\cdot 6^m)$ and $i<j$, then 
$\sup\{x\in B(v')\}\leq i\cdot 6^m+6\cdot 6^m=sup\{x\in B(v)\}+4\cdot 6^m$. 
Also, for any radial edge $vv'$ with $B(v)=B(i\cdot 6^{k-1}, 2\cdot 6^{k-1})$ and $B(v')=B(j\cdot 6^k, 2\cdot 6^k)$, then $\sup\{x\in B(v)\}\leq \sup\{x\in B(v')\}\leq i\cdot 6^{k-1}+22\cdot 6^{k-1}=sup\{x\in B(v)\}+20\cdot 6^{k-1}$.  

Then, if $m=n-1$, $\sup_{w\in \gamma'}\{x\in B(w)\}\leq 2+\sum_{k=1}^m 20\cdot 6^{k-1}=2+20\frac{6^m-1}{5}=4\cdot 6^m-2<6^n$ leading to contradiction.

If $m<n-1$, $\sup_{w\in \gamma'}\{x\in B(w)\}\leq 2+\sum_{k=1}^m 20\cdot 6^{k-1} + 4\cdot 6^m=2+20\frac{6^m-1}{5}+4\cdot 6^m=8\cdot 6^m-2<6^n$ leading also to contradiction. 

Thus, $\gamma$ is a geodesic. Therefore, for any vertex $w\in C_n$ with $l(w)<0$, if $w$ is a shortcut vertex associated to a shortcut 
$\sigma$, then $\sigma$ is a path from $w$ to some vertex $w'\in C_n$ with $l(w')=0$ and $L(\sigma)=l(w)$. 

Let $v\in C_n$ such that $l(v)=-n+1$ and let $\varepsilon=\frac{n-1}{2}$. Then, for every $w\in C_n$ such that $d_{C_n}(v,w)<\varepsilon$, $l(w)\leq -\frac{n-1}{2}$, and for every shortcut $\sigma$ such that $w$ is a shortcut vertex associated to $\sigma$,  $L(\sigma)\geq \frac{n-1}{2}$. Thus,  $G$ is not $\varepsilon$-densely $(k,m)$-path-chordal for any constants $\varepsilon,k,m$.

\begin{definicion} A metric graph $(G,d)$ is $\varepsilon$-densely $k$-path-chordal if for every cycle $C$ with length $L(C)\geq k$, there exist strict shortcuts $\sigma_1,...,\sigma_r$  such that their associated shortcut vertices define an $\varepsilon$-dense subset in $(C,d_C)$.
\end{definicion}

\begin{teorema}\label{Th: hyp-path} If $G$ is $\delta$-hyperbolic, then $G$ is $\varepsilon$-densely $k$-path-chordal.
\end{teorema}

\begin{proof} Let $\varepsilon=2\delta$ and $k=4\delta$. 

Let $C$ be any cycle with $L(C)= 4\delta$. Then, if there is a shortcut vertex in $C$, we are done. Suppose there are no shortcut vertices in $C$. Therefore, there exist two points $x,y\in C$ such that $d(x,y)=\frac{L(C)}{2}=2 \delta$ and the cycle $C$ is formed by two different geodesics, $\gamma_1,\gamma_2$ joining $x,y$. Consider the bigon given by $\gamma_1\cup \gamma_2$. Since $G$ is $\delta$-hyperbolic, for every point $p$ in $\gamma_1$, $d(p,\gamma_2)<\delta$. However, if $m$ is the middle point in $\gamma_1$, then it is clear that $d_C(m,\gamma_2)=\delta$. Therefore, there must be a shortcut $\sigma$ joining $\gamma_1$ and $\gamma_2$ leading to contradiction.  

Now, let $C$ be any cycle with $4\delta < L(C)$. Let us suppose that there is a point $x\in C$ such that there are no shortcut vertices in $B_C(x,2\delta)$. Let $a,b$ be the points in $C$ such that $d_C(a,x)=2\delta=d_C(b,x)$. Since there are no shortcut vertices in $B_C(x,2\delta)$, the restriction of the cycle, $C\cap B_C(x,2\delta)$, defines two geodesics $[xa]$ and $[xb]$ with length $2\delta$ contained in $C$.

Also, since there are no shortcut vertices in $B_C(x,2\delta)$, for every point $y$ in $C\backslash B_C(x,2\delta)$, the geodesic $[xy]$ either contains $[xa]$ or $[xb]$. In fact, since $C\backslash B_C(x,2\delta)$ is connected, there exist two sequences $p_n,q_n$ with $d_C(p_n,q_n)\leq \frac{1}{n}$ and such that $[xp_n]$ contains $[xa]$ and $[xq_n]$ contains $[xb]$. Then, by compactness, taking a subsequence if necessary, we may assume that $p_n$ and $q_n$ converge to a point $p$ in $C\backslash B_C(x,2\delta)$. Thus, by convergence, there are two geodesics, 
$\gamma_1$, $\gamma_2$ from $x$ to $p$ where $\gamma_1=[xa]\cup [ap]$ and $\gamma_2=[xb]\cup [bp]$. See Figure \ref{Figure: Shortcut}. Notice that $L([ap])=L([bp])=d(xp)-2\delta=:\ell$. Now consider the geodesic triangle defined by 
$[xa]\cup [ap]\cup \gamma_2$ and let $m$ be the middle point in $[ax]$. Since $\gamma_1$ is a geodesic, it is trivial to see that $d(m,p)=\delta+\ell$. Also, since the graph is $\delta$-hyperbolic, there is a point $z$ in $[ap]\cup \gamma_2$ such that $d(m,z)<\delta$. In fact, since there are no shortcut vertices in  $B_C(x,2\delta)$, $z\in [ap]\cup [bp]$. However, this implies that $d(m,p)< \delta + d(z,p) \leq \delta + \ell=d(m,p)$ leading to contradiction. 
\end{proof}

\begin{figure}[ht]
\centering
\includegraphics[scale=0.4]{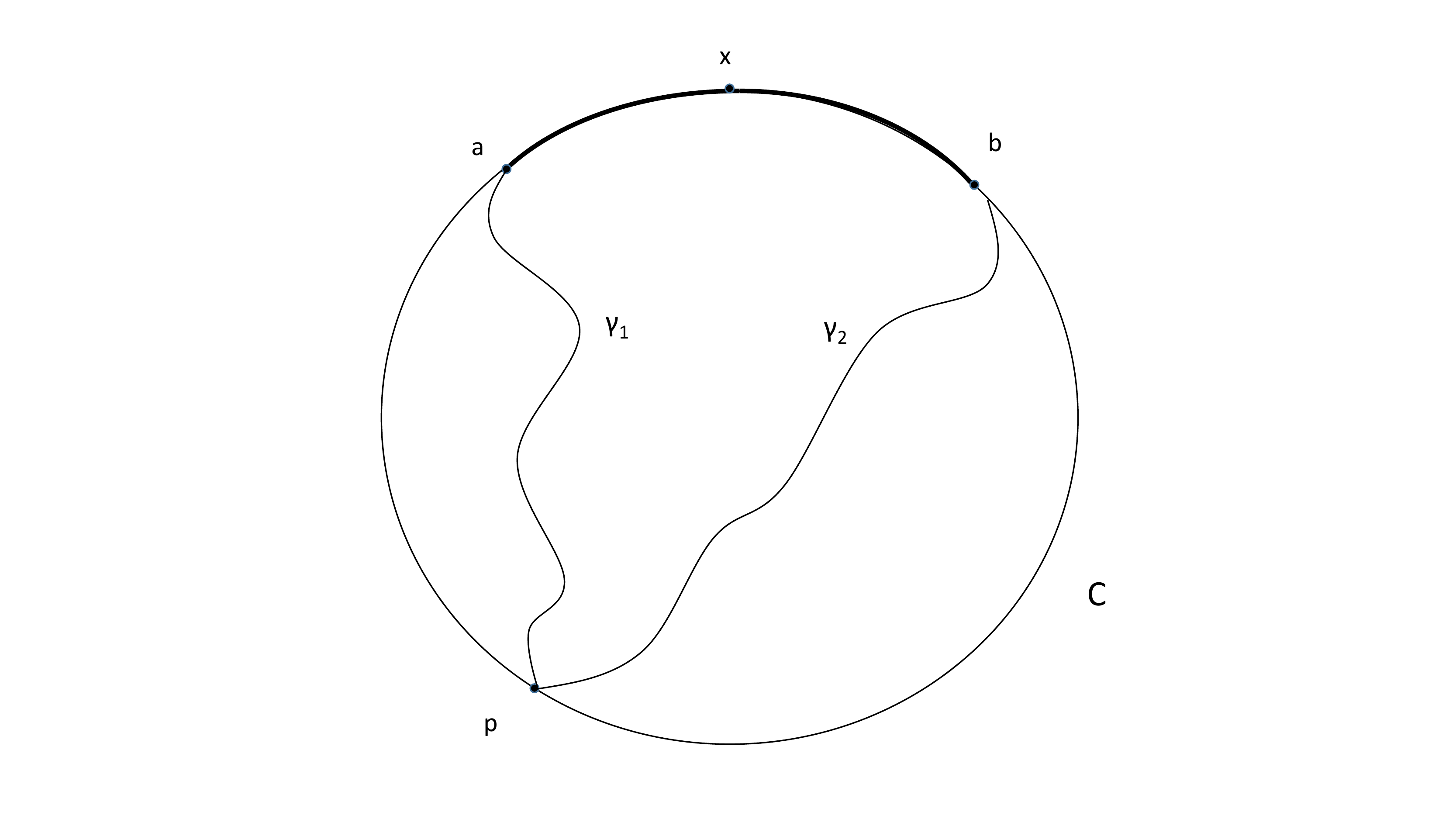}
\caption{If there are no shortcut vertices in $B_C(x,2\delta)$, then there exist a point $p$ and two geodesics $\gamma_1, \gamma_2$ from $x$ to $p$ such that $[xa]\subset \gamma_1$ and $[xb]\subset \gamma_2$.}
\label{Figure: Shortcut}
\end{figure}

\begin{obs} Being  $\varepsilon$-densely $k$-path-chordal does not imply being  $(k,\frac{k}{2})$-path-chordal (this is, $k$-path-chordal as defined in \cite{B}). Also, it does not imply being $\delta$-hyperbolic.

Consider the Cartesian product graph $G=\mathbb{N}\Box \mathbb{N}$ with $L(e)=1$ for every edge $e$ in $G$.

First, notice that $G$ is not $(k,\frac{k}{2})$-path-chordal for any $k\geq 5$. Consider any $n\geq k$ and let $C_n$ be the geodesic square defined by the vertices $(1,1),(n,1),(n,n),(1,n)$. Thus, $C_n$ is a cycle with length $L(C_n)=4n-4>k$. It is trivial to check that every shortcut $\sigma$ in $C_n$ must join two opposite sides of the square and therefore, $L(\sigma)\geq n-1>\frac{k}{2}$.

It is trivial to check that $G$ is not $\delta$-hyperbolic either.

Now, let us see that $G$ is $2$-densely 6-path chordal. Let $C$ be any cycle with $L(C)\geq 6$. Let $(p,q)$ be any vertex in $C$ and 
let us prove that either $(p,q)$ is a shortcut vertex or it is adjacent to a shortcut vertex.

Case 1. Let us suppose that $(p,q)$ and the two adjacent vertices in $C$ are aligned: $(p-1,q)(p,q), (p,q)(p+1,q)\in E(C)$ or 
$(p,q-1)(p,q), (p,q)(p,q+1)\in E(C)$. In both cases, let us see that $(p,q)$ is a shortcut vertex. Suppose $(p-1,q)(p,q)$ and $(p,q)(p+1,q)$ are edges in $C$. It is trivial to see that there exist some $q'\neq q$ such that $(p,q')\in C$. Then, there is a shortcut $\sigma$ in $C$ defined by the geodesic in $G$ joining $(p,q)$ and $(p,q')$. In particular, it is clear that $(p,q)$ is a shortcut vertex since $(p,q)(p,q+1),(p,q)(p,q-1) \notin E(C)$. A similar argument works when $(p,q-1)(p,q)$ and $(p,q)(p,q+1)$ are edges in $C$.

Case 2. Let us suppose that the three vertices are not aligned. Suppose $(p,q-1)(p,q)$ and $(p,q)(p+1,q)$ are edges in $C$ (i.e. $(p,q)$ is an upper-left corner in $C$). (Any other relative position of the adjacent vertices is equivalent to this one up to rotation of the cycle. Therefore, the argument also works.) In this case, $(p-1,q)(p,q)$ and $(p,q)(p,q+1)$ are not edges in $C$. 

Case 2.1. If there is a vertex $(p,q')$ in $C$ with $q'>q$, then the geodesic in $G$ from $(p,q')$ to $(p,q)$ is a shortcut and $(p,q)$ is a shortcut vertex since $(p,q)(p,q+1)\notin E(C)$. Also, if there is a vertex $(p',q)$ in $C$ with $p'<p$, then the geodesic in $G$ from $(p,q')$ to $(p,q)$ is a shortcut and $(p,q)$ is a shortcut vertex since $(p-1,q)(p,q)\notin E(C)$. 

Case 2.2. Suppose that $(p,q)$ is not a shortcut vertex.  
Let us see that either $(p,q-1)$ or $(p+1,q)$ is a shortcut vertex.  If $(p,q-2)(p,q-1)\in E(C)$, then $(p,q-1)$ is a shortcut vertex as we saw in Case 1 and we are done. Also, if $(p+1,q),(p+2,q)\in E(C)$, $(p+1,q)$ is a shortcut vertex and we are done. Otherwise, let us suppose that $(p,q-2)(p,q-1)$ and $(p+1,q),(p+2,q)$ and not edges in $C$. Clearly, $(p,q-1)(p+1,q-1)$ and $(p+1,q)(p+1,q-1)$ can not be simultaneously edges in $C$ since this would induce a cycle of length 4. Therefore, either $(p-1,q-1)(p,q-1)$  or $(p+1,q)(p+1,q+1)$ is an edge in $C$. Suppose $(p-1,q-1)(p,q-1)\in E(C)$ . Then, there is, necessarily, a vertex $(p,q')$ in $C$ with $q'\neq q,q-1$. Since $(p,q)$ is not a shortcut vertex, we can assume by Case 2.1 that $q'<q-1$. Hence, the geodesic in $G$ joining $(p,q-1)$ and $(p,q')$ is a shortcut and $(q,q-1)$ is a shortcut vertex since $(p,q-1)(p,q-2)\notin E(C)$. If $(p+1,q)(p+1,q+1)\in E(C)$ the argument is similar. In this case, there is a vertex $(p',q)$ in $C$ with  $p'>p+1$, the geodesic in $G$ joining $(p+1,q)$ with $(p',q)$ is a shortcut and $(p+1,q)$ is a shortcut vertex since $(p+1,q)(p+2,q)\notin E(C)$ finishing the proof.
\end{obs}

\begin{obs} Being $k$-path-chordal does not imply being $\varepsilon$-densely $k$-path-chordal.

Consider the graph $G_n$ defined as follows. Let $V(G_n):=\{0,1,...,n\}\times \{0,1,...,n\}$. Now, for every $0\leq p \leq n$ and every $1\leq q \leq n$, $(p,q-1)(p,q)\in E(G_n)$. Also, $(p-1,0)(p,0),(p-1,n)(p,n)\in E(G_n)$ for every $1\leq p\leq n$. Finally, if $p$ is odd, for every $1\leq  q \leq n-1$, $(p-1,q)(p,q)\in E(G)$ if and only if $q\equiv 0 \, mod(4)$ and if $p$ is even, for every $1\leq  q \leq n-1$, 
$(p-1,q)(p,q)\in E(G)$ if and only if $q\equiv 2\, mod(4)$.

Then, consider the graph $\mathbb{N}$ and let $G$ be the graph obtained by attaching to each vertex $n$ of 
$\mathbb{N}$ the vertex $(0,0)$ of $G_n$. Suppose every edge in $G$ has length 1. 

Let us see that the resulting graph $G$ is $11$-path-chordal. Let $C$ be any cycle in $G$ with $L(C)\geq 11$. Clearly, by construction, 
$C$ is contained in some $G_n$. Let $P_1:G_n \to \{0,1,...,n\}$ be such that $P_1(p,q)=p$. If $P_1(C)$ contains more than two points, say $p-1,p,p+1$, then there exist two vertices  $(p,q),(p,q')\in C$ with $q'\neq q$. Since $(p,q-1)(p,q)\in E(G_n)$ for every $1\leq q \leq n$, then the geodesic $[(p,q)(p,q')]$ defines a shortcut. On the other hand, if $P_1(C)$ contains only two vertices, say $p-1,p$, since $L(C)\geq 11$ there is some $q$ such that $(p-1,q)(p,q)$ is a shortcut in $C$.

Let us see that $G$ is not  $\varepsilon$-densely $k$-path-chordal for any $k>10$. Suppose $n\geq \max\{k,2\varepsilon\}$ and consider the cycle in $G_n$ given by the geodesic square with vertices $(0,0),(n,0), (n,n)$ and $(0,n)$. It suffices to check that the vertex $(0,q)$ is not a shortcut vertex for any $0\leq q \leq n$. First, notice that for every path $\sigma_1$ from $(0,q)$ to a vertex $(p,0)$ (respectively, a path $\sigma_2$ from $(0,q)$ to a vertex $(p,n)$), $L(\sigma_1)\geq p+q=d_C\big((0,q),(p,0)\big)$ (resp. $L(\sigma_2)\geq p+n-q=d_C\big((0,q),(p,n)\big)$) and therefore $\sigma_1$ (resp. $\sigma_2$) is not a shortcut. Now, consider any path $\sigma$ joining $(0,q)$ and $(n,q')$ with $0< q' <n$. 
If $\sigma$ contains any vertex $(p,0)$ or $(p,n)$ with $0<p<n$, then it is not a shortcut since there are no shortcuts from $(0,q)$ to $(p,0)$ or $(p,n)$ and from $(p,0)$ or $(p,n)$ to $(n,q')$.
Therefore, let us suppose that 
$\sigma\cap C=\{(0,q),(n,q')\}$. Also, notice that the distance between any pair of (consecutive horizontal) edges $(p-1,q_1)(p,q_1)$ and 
$(p,q_2)(p+1,q_2)$ for any $1\leq p\leq n$ and any $0< q_1,q_2 <n$ is at least 2. Then,  $L(\sigma)\geq n+2(n-1)>2n\geq d_C((0,q),(n,q'))$ and $\sigma$ is not a shortcut.
\end{obs}

\begin{obs} Being $(k,\frac{k}{2})$-path-chordal does not imply being $\varepsilon$-densely $k$-path-chordal. Also, it does not imply being $\delta$-hyperbolic.

Consider the graph $G_n$ defined as follows. Let $V(G_n):=\{0,1,...,n\}\times \{0,1,...,n\}$. Now, for every $0\leq p \leq n$ and every $1\leq q \leq n$, $(p,q-1)(p,q)\in E(G_n)$. Also, $(p-1,0)(p,0),(p-1,n)(p,n)\in E(G_n)$ for every $1\leq p\leq n$. Finally, if $p$ is odd, for every $1\leq  q \leq n-1$, $(p-1,q)(p,q)\in E(G)$ if and only if $q\equiv 0 \, mod(8)$ and if $p$ is even, for every $1\leq  q \leq n-1$,
$(p-1,q)(p,q)\in E(G)$ if and only if $q\equiv 4 \, mod(8)$. 

Let $S_n:=\sum_{k=0}^n\frac{1}{2^{k}}$ (notice that $1\leq S_n <2$ for every $n\in \mathbb{N}$) and let $L\big((p,q)(p+1,q)\big)=S_{p+q}=L\big((p,q)(p,q+1)\big)$.

Then, let us consider the graph $\mathbb{N}$ where every edge has length 1 and let $G$ be the graph obtained by attaching to each vertex $n$ in $\mathbb{N}$ the vertex $(0,0)$ of $G_n$. 

First, let us see that $G$ is $(36,18)$-path-chordal. Let $C$ be any cycle of $G$ with $L(C)\geq 36$.  Notice that each cycle is contained in some $G_n$. Moreover, it has a vertex $v=(p+1,q+1)$ such that $(p+1,q), (p,q+1)\in E(G)$ and both edges are contained in $C$ (i.e., $v$ is an upper-right vertex of $C$). Suppose also that $p$ is maximal. By construction, there is an horizontal edge in $G_n$, $(p,q')(p+1,q')$, with $0\leq q-q'\leq 7$. Therefore,  $[(p,q+1)(p,q')]\cup (p,q')(p+1,q')$ defines a path $\gamma$ in $G_n$ with $L(\gamma)<18$. It is immediate to check that this path contains a shortcut $\sigma$ with $L(\sigma)<18$.

Let us see that $G$ is not $\varepsilon$-densely $k$-path-chordal. Consider the cycle $C_n$ in $G_n$ defined by the geodesic square with vertices $(0,0),(n,0),(n,n)(0,n)$. First, notice that for every vertex $(0,q)$ with $0\leq q \leq n$ and every vertex $(p,0)$ with $0\leq p \leq n$, the unique geodesic in $G$ from $(0,q)$ to $(p,0)$ is given by the sequence $(0,q),(0,q-1),...,(0,0),(1,0),...,(p,0)$ and it is contained in $C_n$. Also, for every vertex $(0,q)$ with $0\leq q \leq n$ and every vertex $(p,n)$ with $0\leq p \leq n$, the path in $C_n$ given by the sequence $(0,q),(0,q+1),...,(0,n),(1,n),...,(p,n)$ is geodesic. Therefore, there are no shortcuts in $C_n$ from $(0,q)$ to any vertex $(p,0)$ or $(p,n)$ with $0\leq p \leq n$.

Given any $\varepsilon>0$ and $k\geq 36$ suppose any $n>k,2\varepsilon$ and any vertex $v$ in $C_n$ such that 
$v\in B_C(\frac{n}{2},\varepsilon)$. Since $n\geq 2\varepsilon$, $v=(0,q)$ for some $0<q<n$. 
Let us see that $(0,q)$ is not a shortcut vertex. 

Suppose $\sigma$ is a shortcut and $(0,q)$ is a shortcut vertex associated to $\sigma$. Since there are no shortcuts in $C_n$ from $(0,q)$ to any vertex $(p,0)$ or $(p,n)$,  we may assume that  $\sigma\cap C=\{(0,q),(n,q')\}$ for some $0<q'<n$. However,  by construction of $G_n$, this implies that $L(\sigma)\geq 4(n-1)+n>4n>d_C((0,q),(n,q'))$, leading to contradiction. Thus, there are no shortcut vertices in $ B_C(\frac{n}{2},\varepsilon)$.
\end{obs}

\section{A characterization of $\delta$-hyperbolic graphs}

\begin{definicion} A metric graph $(G,d)$ is \emph{$\varepsilon$-densely $(k,m)$-path-chordal on the triangles} if for every geodesic triangle $T$ with length $L(T)\geq k$, there exist shortcuts $\sigma_1,...,\sigma_r$ with $L(\sigma_i)\leq m$ and such that their associated shortcut vertices define an $\varepsilon$-dense subset in $(T,d_T)$.
\end{definicion}

\begin{teorema}\label{Th: carac} $G$ is $\delta$-hyperbolic if and only if $G$ is $\varepsilon$-densely $(k,m)$-path-chordal on the triangles.
\end{teorema}

\begin{proof} Suppose that $G$ is $\varepsilon$-densely $(k,m)$-path-chordal on the triangles. Let us see that $\delta(G)\leq \max\{\frac{k}{4},\varepsilon+m\}$. Consider any geodesic triangle $T=\{x,y,z\}$. If $L(T)<k$, it follows that every side of the triangle has length at most $\frac{k}{2}$. Therefore, the hyperbolic constant is at most $\frac{k}{4}$. Then, let $L(T)\geq k$ and let us prove that $T$ is $(\varepsilon+m)$-thin. Consider any point $p\in T$ and let us assume that $p\in [xy]$. If $d(p,x)< \varepsilon+m$ or $d(p,y)< \varepsilon+m$, we are done. Otherwise, there is a shortcut vertex $x_i$ such that $d(x_i,p)<\varepsilon$ and a shortcut $\sigma_i$, with $x_i\in \sigma_i$ and $L(\sigma_i)\leq m$. Since $[xy]$ is a geodesic, $\sigma_i$ does not connect two points in $[xy]$ and $d(p, [xz]\cup [yz])<\varepsilon+m$.

Suppose that $G$ is $\delta$-hyperbolic and consider any geodesic triangle $T=\{x,y,z\}$ with $L(T)\geq 8\delta$. Let $p\in T$ and let us assume, with no loss of generality, that $p\in [xy]$. Since $G$ is $\delta$-hyperbolic, $d(p, [xz]\cup [yz])<\delta$. If $d(px),d(py)>\delta$, then there is a path $\gamma$ with $L(\gamma)<\delta$ joining $p$ to $[xz]\cup [yz]$. In particular, there is a shortcut $\sigma\subset \gamma$ with $L(\sigma)\leq L(\gamma)<\delta$ joining some shortcut vertex $p'\in [xy]$ with $d(p,p')<\delta$ to $[xz]\cup [yz]$. Therefore, if $L([x,y])>2\delta$, for every point $q\in [x,y]$ there is a shortcut vertex in $B_T(q,2\delta)\cap [x,y]$ associated to a shortcut with length at most $\delta$. Since $L(T)\geq 8\delta$, there is at most one side of the triangle with length $\leq 2\delta$. Then, for every point $p$ in the triangle there is a shortcut vertex in $B_T(p,3\delta)$ associated to a shortcut with length at most $\delta$. Thus, it suffices to consider $\varepsilon=3\delta$,  $k=8\delta$ and $m=\delta$.
\end{proof}

\end{document}